\documentclass[a4paper,12pt]{amsart}
\usepackage{amsmath}
\usepackage{comment}
\usepackage{amssymb}
\usepackage{subcaption}
\usepackage{ifthen}
\usepackage{graphicx}
\usepackage{float}
\usepackage{cite}
\usepackage{amsfonts}
\usepackage{amscd}
\usepackage{amsxtra}
\usepackage{color}
\usepackage{enumerate}
\usepackage{enumitem}
\usepackage{url}

\numberwithin{equation}{section}
\newtheorem{theorem}{Theorem}[section]
\newtheorem{lemma}[theorem]{Lemma}

\newtheorem{definition}[theorem]{Definition}

\theoremstyle{definition}
\newtheorem{remark}[theorem]{Remark}

\newtheorem{example}[theorem]{Example}

\numberwithin{equation}{section}

\newcounter{alphabet}
\newcounter{tmp}
\newenvironment{Thm}[1][]{\refstepcounter{alphabet}%
	\bigskip%
	\noindent%
	{\bf Theorem \Alph{alphabet}}%
	\ifthenelse{\equal{#1}{}}{}{ (#1)}%
	{\bf .} \itshape}{\vskip 8pt}



\begin{document}
	\title[REVISIT OF MEROMORPHIC CONVEX FUNCTIONS]
	{REVISIT OF MEROMORPHIC CONVEX FUNCTIONS}

	\author[Vibhuti Arora]{Vibhuti Arora}
	\address{Vibhuti Arora, Department of Mathematics, National Institute of Technology Calicut, 673 601, India.}
	\email{vibhutiarora1991@gmail.com, vibhuti@nitc.ac.in}

	\author[Vinayak M.]{Vinayak M.}
	\address{Vinayak M., Department of Mathematics, National Institute of Technology Calicut, 673 601, India.}
	\email{vinayak\_p220286ma@nitc.ac.in, mvinayak.math@gmail.com}
	
	\newcommand{\D}{{\mathbb D}}
	\newcommand{\C}{{\mathbb C}}
	\newcommand{\real}{{\operatorname{Re}\,}}
	\newcommand{\Log}{{\operatorname{Log}\,}}
	\newcommand{\Arg}{{\operatorname{Arg}\,}}
	\newcommand{\ds}{\displaystyle}
	
	
	\begin{abstract}
			Our primary aim is to explore a sufficient condition for the class of meromorphically convex functions of order $\alpha$, where $0 \leq \alpha < 1$. The investigation will focus on studying a class of continuous functions defined on 
$[0,1)$, and analyzing the properties of the Schwarzian norm of locally univalent meromorphic functions. Moreover, a new subclass of meromorphic functions is also introduced, and some of its characteristics are examined.
	\end{abstract}
	
	\maketitle
\section{Introduction}
We begin with the definition of the class $\mathcal{A}$ which consists of analytic functions $f$ defined on $\mathbb{D}:=\{z \in \mathbb{C}: |z|<1\}$  satisfying the normalization condition $ f(0)=0 \text{ and } f'(0)=1.$ Due to the well-known Reimann mapping theorem, the univalent function theory plays an important role in understanding the geometric behaviour of analytic functions. The famous de Branges theorem \cite{DE85}, which gives a necessary condition for the univalency of analytic functions, attracted many researchers. Later, finding the necessary and sufficient conditions for functions to be in the class $\mathcal{S}$ and its subclasses became a major area of research in this field. 
The {\em class of convex functions} $\mathcal{C}$ ({\em  starlike functions} $\mathcal{S}^*$) containing functions in $\mathcal{S}$ whose image is a convex set (starlike set with respect to $0$), 
are some of the important subclasses of $\mathcal{S}$.
See \cite{ DUR83, POM75, GOO83, TIT39, MAM94} for a detailed exploration of these topics.

\medskip

In 1936, Robertson \cite{ROB36} introduced a subclass of $\mathcal{C}$ and $\mathcal{S}^*$ as follows. For $\alpha \in [0,1]$, a  function $f \in \mathcal{S}$ is said to be {\em convex function of order $\alpha$} if  
$$\real\left(1+\dfrac{zf''(z)}{f'(z)}\right)\geq\alpha,$$
and is said to be {\em starlike function of order $\alpha$} if 
$$\real\left(\dfrac{zf'(z)}{f(z)}\right)\geq\alpha.$$ The respective classes are denoted by $\mathcal{C}(\alpha)$ and $\mathcal{S}^*(\alpha)$. Due to the well-known Alexander characterization \cite[p. 43]{DUR83}, $f \in \mathcal{C}(\alpha)$ if and only if $zf' \in \mathcal{S}^*(\alpha)$. Note that when $\alpha=0$, the above two classes reduce to $\mathcal{C}$ and $\mathcal{S}^*$, respectively. 

\medskip

A function that is analytic except possibly at poles is called {\em meromorphic function}. In fact, every meromorphic function is analytic almost everywhere. Our main interest is to study the class $\mathcal{B}$ of meromorphic functions defined in $\mathbb{D}$ having the form 
\begin{equation}\label{MEROFUNLAURENT}
    f(z) = \dfrac{1}{z}+a_0+a_1z+ \cdots .
\end{equation} 
We denote by $\mathcal{M}$ the class of univalent functions in $\mathcal{B}$.
Note that every function in $\mathcal{B}$ is analytic on $\mathbb{D} \setminus \{0\}$. Also, it is well-known that a non-zero function $f$ is analytic in $\mathbb{D}$ if and only if $1/f$ is meromorphic in $\mathbb{D}$. More precisely, $f \in \mathcal{S}$ if and only if $1/f \in \mathcal{M}$. 

\medskip

The concept of Schwarzian derivative attained considerable attention in finding necessary and sufficient conditions for functions to be in the class $\mathcal{S}$ and its various subclasses. For an analytic locally univalent function $f$ (in other words, $f'(z)\ne 0 $ in $\mathbb{D}$), the function $S_f$ defined by
\begin{equation*}
    S_f(z) :=\left(\dfrac{f''(z)}{f'(z)}\right)' - \dfrac{1}{2} \left(\dfrac{f''(z)}{f'(z)}\right)^2
\end{equation*}
and is called the {\em Schwarzian derivative} of $f$ at the point $z$. The study of $S_f$ is highly significant due to its invariance property under any bilinear transformation. More precisely, if $T$ is any bilinear transformation, then
\begin{equation}\label{SCHWARZINVARIANT}
    S_{T\circ f} = S_f.
\end{equation}
Note that $S_f=0$ if and only if $f$ is a bilinear transformation.
 The norm of $S_f$ can be defined as
\begin{equation*}
    \lVert S_f \rVert := \sup_{z\in \mathbb{D}}(1-|z|^2)^2|S_f(z)|.
\end{equation*}
Several necessary and sufficient conditions for certain classes of functions have been obtained by using the above norm. In 1948, Nehari\cite[Theorem 1]{NEH49} found a sufficient condition which states that if 
$\lVert S_f \rVert \leq 2,$ for analytic function $f$ then $f$ is univalent in $\mathbb{D}$.
Hille \cite{HIL49} showed that the constant $2$ is the best possible. Conversely, in \cite{NEH49}, it has been proved that if $f$ is univalent in $\mathbb{D}$, then $\lVert S_f \rVert \leq 6$, which is also sharp. Later, in \cite{NEH76}, it is shown that $\lVert S_f \rVert \leq 2$ becomes a necessary criterion for a function to be convex. Moreover, the inequality becomes strict for any bounded convex function. A simple proof of this can be found in \cite{CHU11}. Recently, Kim et al. \cite{KIM22} obtained sufficient conditions for a locally univalent function to be starlike and convex in terms of $S_f$. 

\medskip

The invariant nature of the Schwarzian derivative given by the relation \eqref{SCHWARZINVARIANT} helps us to conclude that the estimates based on the Schwarzian derivative can
be extended to the case of meromorphic functions. Note that simple computations demonstrate  $S_f = S_{1/f}$. This naturally leads to the following question;  {\em can concepts like univalency and geometric properties such as convexity and starlikeness be applied to meromorphic functions?}

\medskip

Subsequently, numerous researchers have addressed the question of understanding the analytic behavior of meromorphic functions and their subclasses by estimating their Schwarzian derivative. Yamashita \cite{YAM79} obtained sufficient conditions for the univalency of meromorphic functions in $\mathbb{D}$ by estimating the upper bound for the line integral of the Schwarzian derivative. In 1956, Haimo \cite{HAI56} considered a bound for $S_f$ using the techniques in ordinary linear differential equations and established the following sufficient condition for a function to be {\em meromorphically convex}, i.e. the function which maps $\mathbb{D}$ onto the complement of a convex domain.

\begin{Thm}
    Let $f \in \mathcal{B}$ be a locally univalent function with 
    $$|S_f(z)| \leq 2q(|z|), \quad |z|<1,$$
where $q$ is a positive continuous function on $[0,1)$ having the property that the differential equation
\begin{equation}\label{DIFFEQTHMA}
            y''+qy=0 
        \end{equation}
        has a non vanishing solution in $(0,1)$,  satisfies $y(0)=0$, and
\begin{equation}\label{INEQTHMA}
    \lim _{x\to 1}\dfrac{xy'(x)}{y(x)} \geq \dfrac{1}{2}.
\end{equation}
        Then $f$ maps $\mathbb{D}$ onto the complement of a convex domain. The result is sharp in the sense that if $q$ is analytic in $\mathbb{D}$ with $|q(z)| \leq q(|z|)$, then the constant $2$ is the sharpest one.
\end{Thm}

In the proof of the above theorem by Haimo, the sharpness part had some issues. Recently, Chuaqui et al. \cite{CHUA11} studied the sufficient criterion for the meromorphically convex function, which is similar to Theorem A. They modified the sharpness criterion by resolving the issue in the paper \cite{HAI56} and also obtained a considerably easier proof. 
 The sufficient conditions for meromorphically convex functions of order $\alpha$ are studied by Arora et al. in \cite{ARO18}. For more articles regarding this, we refer \cite{CLU59, MIL70, NUN01, POM63, TOT10, MIL00}.

\medskip

This paper aims to extensively explore the class $\mathcal{B}$ of meromorphic functions. 
 In our study, we use the ideas from the proof of theorems mentioned above from \cite{HAI56, CHUA11} to establish a sufficient condition in terms of the Schwarzian derivative for the class of meromorphic convex functions of order $\alpha$, which we denote by $\mathcal{BC}(\alpha)$. Our study also introduces a subclass of meromorphic functions, namely the class of meromorphic inverse convex functions of order $\alpha$, denoted by $\mathcal{BC}_I(\alpha)$. We try to explore some of the inclusion properties and provide an analytic characterization of the class $\mathcal{BC}_I(\alpha)$. Additionally, we establish a connection between $\mathcal{BC}_I(\alpha)$ and the class of meromorphic starlike function of order $\alpha$,  $\mathcal{BS}^*(\alpha)$. Furthermore, the radius of inverse conversity for the functions in the class $\mathcal{B}$ is obtained. 

\section{Definitions and main results} 

This section deals with subclasses of the class 
$\mathcal{B}$, namely the class of meromorphically convex functions of order $\alpha$ and meromorphically inverse convex functions of order $\alpha$.

\subsection{Meromorphically convex functions of order $\alpha$ }
 Analogous to the characterization for an analytic convex function of order $\alpha$, a meromorphic function of the form (\ref{MEROFUNLAURENT}) is defined to be convex of order $\alpha$ as follows (see \cite[p. 235]{GOO83}).

\begin{definition}\label{CONVEXMERO} 
  Let $\alpha \in [0,1)$. A function $f \in \mathcal{B}$ with $f'(z)\neq0$ in $\mathbb{D}\setminus \{0\}$ is said to be meromorphically convex of order $\alpha$ if   
  \begin{equation*}
      \real\left(1+ \dfrac{zf''(z)}{f'(z)}\right)\leq-\alpha.
  \end{equation*}
  We denote the class of such functions by $\mathcal{BC}(\alpha)$.
\end{definition}
Similarly, we can define meromorphically starlike function of order $\alpha$.
\begin{definition}\label{STARLIKEMERO}
  Let $\alpha \in [0,1)$. A function $f \in \mathcal{B}$ with $f(z) \neq 0$ in $\mathbb{D} \setminus \{0\}$ is said to be meromorphically starlike of order $\alpha$ if   
 \begin{equation*}
      \real\left( \dfrac{zf'(z)}{f(z)}\right)\leq-\alpha.
  \end{equation*}
  We denote the class of such functions by $\mathcal{BS}^*(\alpha)$.
\end{definition}
\begin{remark}
    For $\alpha=0$, the above two classes in  Definitions \ref{CONVEXMERO} and \ref{STARLIKEMERO} will reduce to the class of meromorphically convex and starlike functions, respectively (see  \cite[p. 47]{POM75}).  Moreover, in this case, $f$ maps $\mathbb{D}$ onto the complement of a convex domain (starlike domain with respect to 0, respectively).
\end{remark}
 The following is an example of a meromorphically convex function of order $1/2$.
\begin{example}
    Define the meromorphic function 
    $$f(z):=\dfrac{z}{4}+\dfrac{1}{z}, \, z \in \mathbb{D},$$
    which is clearly univalent. Simple computations give
    $$\real\left(1+\dfrac{zf''(z)}{f'(z)}\right) = \real\left(\dfrac{z^2+4}{z^2-4}\right)<-\dfrac{1}{2}.$$
    That is, $f$ is meromorphically convex of order 1/2.
    
\end{example}
Now, we are interested in obtaining a sufficient condition for a meromorphic function to be in the class $\mathcal{BC}(\alpha)$. Surprisingly, most of the necessary conditions for functions to be univalent, starlike, convex, etc., are proved using the fundamental results in the theory of complex-valued functions but sufficient conditions are obtained using the
theory of ordinary differential equations. Thus,
the following class of functions introduced in \cite{CHUA11} is an integral part of our study.
\begin{definition}
   For $\alpha \geq 0$, the class $\mathcal{P}(\alpha)$  is defined as the class of continuous functions $q_\alpha$ on $[0,1)$ with the following properties
    \begin{enumerate}[label=(\alph*)]
        \item For $r \in [0,1), \, q_\alpha(r) \geq 0.$
        \item For the differential equation
        \begin{equation}\label{DIFFEQ}
           y''+q_\alpha y=0 
        \end{equation}
        satisfying $y(0)=0$ and $y'(0)=1$, the solution is positive on $(0,1)$ with 
        \begin{equation*}
            \lim _{x \to 1^-} \dfrac{y'(x)}{y(x)} \geq \alpha.
        \end{equation*}
    \end{enumerate}
\end{definition}

    \begin{remark}
         From the definition itself, it is obvious that the class $\mathcal{P}(\alpha)$ satisfies the inclusion 
    $$\mathcal{P}(\alpha) \subset \mathcal{P}(\beta), \text{ for } \alpha>\beta.$$ 
    
    \end{remark}

    Next, we state our first main result, which gives a sufficient condition for a meromorphic function to be in the class $\mathcal{BC}(\alpha).$

\begin{theorem}\label{MAINT1}
Let $\alpha \in [0,1)$ and $q_\alpha \in \mathcal{P}((1+\alpha)/2)$. Assume  $f \in \mathcal{B}$ be a locally univalent function satisfying the condition  $|S_f(z)| \leq 2 q_\alpha(|z|)$. Then $f \in \mathcal{BC}(\alpha)$. Also, the constant $(1+\alpha)/2$ is the best possible.
\end{theorem}

\begin{remark}
    In particular for a fixed $\alpha \in [0,1)$, if we take $q_\alpha(z)=c_\alpha,$ a positive constant, then $q_\alpha \in \mathcal{P}(\alpha)$ if and only if the solution of the differential equation
    \begin{equation}\label{DIFFEQREMARK}
        y''+c_\alpha y=0, \, y(0)=0, \,  y'(0)=1
    \end{equation} 
    satisfies $y(x)>0$ in $(0,1)$ with 
    \begin{equation}\label{LIMEQREMARK}
        \lim_{x \to 1^-} \dfrac{y'(x)}{y(x)} =\dfrac{1+\alpha}{2}.
    \end{equation} 
Note that for the differential equation \eqref{DIFFEQREMARK}, the root is $$y(x)=\dfrac{\sin \sqrt{c_\alpha}x}{\sqrt{c_\alpha}}.$$ 
Next, we need to check whether the function $y$  satisfies \eqref{LIMEQREMARK}. It is simple to observe that
\begin{equation*}
        \lim_{x \to 1} \dfrac{y'(x)}{y(x)} = \lim_{x\to1^-} \sqrt{c_\alpha} \cot \sqrt{c_\alpha}x=\sqrt{c_\alpha} \cot \sqrt{c_\alpha}.
\end{equation*}
  It leads to
   $$\lim_{x \to 1} \dfrac{y'(x)}{y(x)} =\dfrac{1+\alpha}{2} \Leftrightarrow \sqrt{c_\alpha}-\dfrac{1+\alpha}{2} \tan \sqrt{c_\alpha}=0,$$
 and so in this case, Theorem \ref{MAINT1} coincides with \cite[Theorem $2.1$]{ARO18}.
\end{remark}

Next, we will illustrate Theorem \ref{MAINT1} using the following example.

\begin{example}
    Consider the function $f$ defined on $\mathbb{D}$, given by
    \begin{equation*}
        f(z) = \sqrt{\dfrac{1-\alpha}{\pi}} \cot \left( \sqrt{\dfrac{1-\alpha}{\pi}}z\right ).
    \end{equation*}
    As $\cot z$ is a meromorphic function with a simple pole at $0$, $f$ has the series expansion given by \eqref{MEROFUNLAURENT}. It is easy to verify that
    \begin{equation*}
        S_f(z) = \dfrac{2-2\alpha}{\pi} \left(\csc ^2\left( \sqrt{\dfrac{1-\alpha}{\pi}}z\right ) - \cot^2 \left( \sqrt{\dfrac{1-\alpha}{\pi}}z\right ) \right) = \dfrac{2-2\alpha}{\pi} =  2 \left(\dfrac{1-\alpha}{\pi}\right). \end{equation*}
        Next, we define the function $q_\alpha$ by 
         \begin{equation*}
            q_\alpha(x) = \dfrac{2-2\alpha}{\pi (1+x^2)}, ~ x \in [0,1).
        \end{equation*}
         Now note that for all $r \in (0,1)$, we have $\pi (1+r^2) \leq 2\pi$, and hence we deduce that
        \begin{equation*}
            \dfrac{2-2\alpha}{\pi (1+r^2)} \geq \dfrac{1-\alpha}{\pi}. 
        \end{equation*}
       This shows that
               \begin{equation*}
            |S_f(z)| = \dfrac{4-4\alpha}{\pi} \leq   \dfrac{4-4\alpha}{\pi (1+|z|^2)}=2q_\alpha(|z|), ~ z \in \mathbb{D}.
        \end{equation*}
        It is easy to verify that
        \begin{equation*}
            \int_0^1 q_\alpha(x) dx = \dfrac{2-2\alpha}{\pi} \tan^{-1}(1) = \dfrac{1-\alpha}{2}.
        \end{equation*}
Then by Lemma \ref{OURLEMMA2}, we get $q_\alpha \in \mathcal{P}(1+\alpha)/2.$ Observe that $f$ satisfies all the hypothesis in Theorem \ref{MAINT1} and as a result, we conclude that $f \in \mathcal{BC}(\alpha)$.

\medskip

Next, by the definition of the class $\mathcal{BC}(\alpha)$, we are going to show that $f \in \mathcal{BC}(\alpha)$. For that consider
\begin{equation*}
    f(z) = b_\alpha \cot \left(b_\alpha  z\right) =  b_\alpha \dfrac{u(z)}{v(z)},  
\end{equation*}
where
\begin{equation*}
  b_\alpha = \sqrt{\dfrac{1-\alpha}{\pi}}, ~  u(z) = \cos \left(b_\alpha  z\right), ~  ~ v(z) = \sin \left(b_\alpha z\right).
\end{equation*}
It is easy to verify that $u ~ \text{and} ~ v$ satisfies
\begin{equation*}
    u(0)=1, ~ u'(0)=0 ~ \text{and} ~ v(0)=0, ~ v'(0)=1,
\end{equation*}
which implies that the
Wronskian $W(u,v)(z) =1$ for all $z \in \mathbb{D}$. Simple computations give
\begin{align*}
    \real \left (1+\dfrac{zf''(z)}{f'(z)} \right)&= 1-2 \real \left(\dfrac{zv'(z)}{v(z)}\right) = 1- 2  \real \left(\dfrac{zb_\alpha \cos(b_\alpha z)}{\sin(b_\alpha z)}\right).
\end{align*}
By Lemma \ref{LEMMAFORFUNCTIONRATIO}, it is enough to prove 
\begin{equation}\label{IN1}
   \real \left(\dfrac{zb_\alpha \cos(b_\alpha z)}{\sin(b_\alpha z)}\right) \geq \dfrac{1+\alpha}{2}=\dfrac{2-\pi b_\alpha^2}{2}.
\end{equation}
By taking $z=x+iy \in \mathbb{D}$ with $x \geq 0, ~ y \geq 0$ and performing a straightforward computation, one can see that 
\begin{equation*}
     \real \left(\dfrac{zb_\alpha \cos(b_\alpha z)}{\sin(b_\alpha z)}\right)=\dfrac{b_\alpha x \cos \left(b_\alpha x \right) \sin \left(b_\alpha x \right)+b_\alpha y \sinh \left(b_\alpha y \right)\cosh \left(b_\alpha y \right) }{\sin^2\left(b_\alpha x\right)\cosh^2\left(b_\alpha y\right)+\sinh^2\left(b_\alpha y\right)\cos^2\left(b_\alpha x\right)}.
\end{equation*}
Then the inequality \eqref{IN1} will be equivalent to
\begin{align*}
    2b_\alpha x \cos(b_\alpha x)\sin(b_\alpha x) - (2-\pi b_\alpha^2) \cosh^2(b_\alpha y ) \sin(b_\alpha x) \cos (b_\alpha x) \tan (b_\alpha x) \\ \geq (2-\pi b_\alpha^2)  \cos^2(b_\alpha x) \sinh(b_\alpha y) \cosh(b_\alpha y) \tanh(b_\alpha y)-2b_\alpha y \sinh(b_\alpha y) & \cosh(b_\alpha y).
\end{align*}
That is,
\begin{align*}
    \cos(b_\alpha x)\sin(b_\alpha x) [2b_\alpha x -(1+\alpha) \cosh^2(b_\alpha y )\tan (b_\alpha x)] \\\geq\sinh(b_\alpha y) \cosh(b_\alpha y) [(1+\alpha) \cos^2(b_\alpha x)&\tanh(b_\alpha y)-2b_\alpha y].
\end{align*}
From the fact that $-\cosh^2(b_\alpha y)\tan(b_\alpha x) \leq -\tan(b_\alpha x)$ and $(1+\alpha)\tan (b_\alpha x) \geq \tan (b_\alpha x)$, we get
\begin{align}\label{IN2}
  \nonumber  \cos(b_\alpha x)\sin(b_\alpha x) [2b_\alpha x -\tan (b_\alpha x)] \\
    \geq \sinh(b_\alpha y) \cosh(b_\alpha y)& [(1+\alpha) \cos^2(b_\alpha x)\tanh(b_\alpha y)-2b_\alpha y].
\end{align}

\medskip

Define the function $g(y):= (1+\alpha)\tanh(b_\alpha y)-2b_\alpha y, ~ y \geq 0$. We have $\operatorname{sech}^2(y) \leq 1, ~\text{for all} ~ y \geq 0$, which implies that
$$g'(y) = (1+\alpha)b_\alpha \operatorname{sech}^2 (b_\alpha y)-2b_\alpha = b_\alpha[(1+\alpha)\operatorname{sech}^2 (b_\alpha y)-2]  <b_\alpha (-1+\alpha)<0.$$
This means that $g$ is a decreasing function in $[0,\infty)$ and as a result we have $g(y) \leq g(0) = 0.$
Since $\cos^2(b_\alpha x) \leq 1$, we have
$$(1+\alpha) \cos^2(b_\alpha x)\tanh(b_\alpha y)-2b_\alpha y \leq (1+\alpha) \tanh(b_\alpha y)-2b_\alpha y \leq 0.$$
By the definition itself, for all $x \geq 0$, we have $\sinh(b_\alpha y) \cosh(b_\alpha y) \geq 0.$ So the right-hand side of \eqref{IN2} is either negative or zero. 

\medskip
It is obvious that $2b_\alpha x -\tan (b_\alpha x)> 0$ and $ \cos(b_\alpha x)\sin(b_\alpha x)\geq 0$. Hence, we obtain that the left-hand side of \eqref{IN2} is always non-negative. So we conclude that on the first quadrant of the unit disk, we established the inequality \eqref{IN1}. Due to the symmetry of the unit disk about the real and imaginary axes, we obtained the desired inequality \eqref{IN1}.
\end{example}
\medskip

\subsection{Meromorphically inverse convex functions of order $\alpha$}
We start this section by recalling the following.
Let $f \in \mathcal{S}^*(\alpha)$ and take $g=1/f$. Then clearly $g \in \mathcal{M}$ and 
$$\real \left(\dfrac{zf'(z)}{f(z)}\right) =\real \left(\dfrac{-zg'(z)}{g(z)}\right) \geq \alpha.$$ 
which implies $g\in \mathcal{BS}^*(\alpha)$. That is, there is a one-to-one correspondence between the classes $\mathcal{S}^*(\alpha)$ and $\mathcal{BS}^*(\alpha)$. In other words, 
\begin{equation}\label{fand1/f}
    f \in \mathcal{S}^*(\alpha)  \Leftrightarrow 1/f \in \mathcal{BS}^*(\alpha).
\end{equation}
 However, it is important to note that the same argument is invalid for the case of convexity. The following example precisely discusses this.

\begin{example}
    Consider the function 
$l(z)=z/(1-z), \, z \in \mathbb{D},$ clearly $l\in \mathcal{C}$ but $l\notin \mathcal{C}(\alpha), \alpha\in (0,1)$.
Assume
$g(z)=1/l(z)=(1-z)/z,$
we get
$$\real \left(1+\frac{zg''(z)}{g'(z)}\right)=-1.$$
This implies that $g$ belongs to the class $\mathcal{BC}(\alpha)$ for all $\alpha\leq 1$. In particular, $g \in \mathcal{BC}(1/2)$ but $l$ is not in the class $\mathcal{C}(1/2)$.
\end{example}

A reasonable question is whether the reciprocal of functions in the class $\mathcal{C}(\alpha)$ exhibits any notable analytic properties. In \cite{TOT10}, the case $\alpha=0$ is examined, leading to the introduction of a new class, namely the class of meromorphically inverse convex functions;  which contains functions $f \in \mathcal{B}$ such that $g(z)f(z)=1, ~z \in \mathbb{D} \setminus \{0\}$, for some $g \in \mathcal{C}$.
\begin{definition}\label{DEFALPHAINVCONV}
    A function $f \in \mathcal{B}$ is said to be meromorphically inverse convex of order $\alpha$ in $\mathbb{D}\setminus\{0\}$ if there exists a function $g \in \mathcal{C}(\alpha)$ with
    $g(z)f(z) = 1, \text{ for } z \in \mathbb{D}\setminus\{0\}.$ We denote such class of functions by $\mathcal{BC}_{I}(\alpha)$.
\end{definition}
In this subsection, we discuss some of its analytic properties and also explore its association with existing classes. We also identify the largest disk on which a functions of the class $\mathcal{M}$ is in $\mathcal{BC}_{I}(\alpha)$.

\begin{remark}
    Note that if $f \in \mathcal{BC}_I (\alpha)$, then $f=1/g$, where $g \in \mathcal{C}(\alpha)$. As a result, every function in the class $\mathcal{BC}_I (\alpha)$ is univalent.
\end{remark}

\begin{example}
    Define $g:\mathbb{D}\setminus\{0\} \to\mathbb{C}$ by
    $$g(z)=-\dfrac{1}{\log(1-z)},$$ and  $f(z)=1/g(z).$
    Then
    $$\real\left(1+\dfrac{zf''(z)}{f'(z)}\right) = \real\left(\dfrac{1}{1-z}\right)>\dfrac{1}{2},$$
    which implies $f \in \mathcal{C}(1/2)$ and hence $g \in \mathcal{BC}_{I}(1/2)$.
\end{example}

\begin{example}
    Fix $\alpha \neq 1/2$. Consider the function
    $$g(z) = \dfrac{\eta}{1-(1-z)^\eta },$$
    where $\eta=2\alpha-1,$ and $f(z)=1/g(z)$. Note that
    \begin{equation*}
            \real \left(1+\dfrac{zf''(z)}{f'(z)}\right) =\real \left(1+\dfrac{z(1-\eta)}{1-z}\right)= \real \left(1+ \dfrac{2z(1-\alpha)}{1-z}\right)>\alpha,
    \end{equation*}
    which implies $f \in \mathcal{C(\alpha)}$ and so $g \in \mathcal{BC}_{I}(\alpha).$ 
\end{example}
Now we state the following result, which studies some inclusion properties and a characterization of the class $\mathcal{BC}_I( \alpha)$.
\begin{theorem}\label{MAINT5}
    The class $\mathcal{BC}_{I}(\alpha)$ possesses the following properties.
    \begin{enumerate}[label=(\roman*)]
    \item For $\alpha>\beta$, $\mathcal{BC}_I(\alpha) \subset \mathcal{BC}_I(\beta).$ 
        \item $\mathcal{BC}_I( \alpha)$ is invariant under any non zero linear transformation $T(z) = \lambda z$.
        
        \item For any $\alpha \in [0,1)$, the class $\mathcal{BC}_{I}(\alpha)$ is properly contained in $\mathcal{BS}^*(0).$

        \item $g \in \mathcal{BC}_{I}(\alpha)$ if and only if $g$ is univalent and 
\begin{equation}\label{EQT3}
    \real\left(1+\dfrac{zg''(z)}{g'(z)}-\dfrac{2zg'(z)}{g(z)}\right) \geq \alpha.
\end{equation}
       
    \end{enumerate}
\end{theorem}
Next, we state the following result, which gives a radius for which every function in $\mathcal{M}$ is to be in $\mathcal{BC}_{I}(\alpha)$. 
\begin{theorem}\label{MAINT2}
Let $g \in \mathcal{M}$. Then $g$ is meromorphically inverse convex of order $\alpha$ on the disk of radius $r_\alpha$, where $r_\alpha$ is the unique zero in $(0,1)$ of the polynomial
\begin{equation}\label{POLYRADIUS}
    P_\alpha(x):=(-1-\alpha)x^2+4x+\alpha-1.
\end{equation}
This is false for every $r>r_\alpha$.
\end{theorem}

Finally, we  obtain the following relation between the classes $\mathcal{BC}_{I}(\alpha)$ and $\mathcal{BS}^*(\alpha)$. 
\begin{theorem}\label{MAINT4}
   A function  $g \in \mathcal{BC}_{I}(\alpha)$ if and only if $\dfrac{1}{z(1/g)'} \in \mathcal{BS}^*(\alpha)$.
\end{theorem}
\section{Preliminary results}

 We have the following result from \cite[p. 259]{DUR83}, which establishes a connection between the Schwarzian derivative of complex-valued functions and a second-order ordinary linear differential equation.

\begin{lemma}\label{DURENTHEOREM1}
    Let $p$ be an analytic function and $f$ be a function with the Schwarzian derivative $S_f=2p$. Then $f$ is of the form 
    \begin{equation}\label{U/V}
        f(z) = \dfrac{u(z)}{v(z)},
    \end{equation}
where $u$ and $v$ are arbitrary linearly independent solutions of the ordinary linear differential equation $y''+py=0
$.
\end{lemma}

\medskip

By adopting a similar method of proof as outlined in \cite[Lemma II]{ HAI56},  we arrive at the following result.
\begin{lemma}\label{OURLEMMA1}
    Let $\alpha \in [0,1)$ and $g$ be a continuously differentiable function on $[0,1)$ satisfying the conditions $g(0)=0$,   $g'(0)\neq 0$. Then for given $\epsilon>0$, there exists a $\delta>0$ such that for any $r \in (1-\delta,1)$,
    $$r\int_0^r (g'(x))^2 dx \geq r \int_0^r q_\alpha(x)g^2(x) dx + \left(\dfrac{1+\alpha}{2}-\epsilon\right) g^2(r),$$
   where $q_\alpha \in \mathcal{P}((1+\alpha)/2).$  
\end{lemma}

\begin{proof}
    Let $y$ be the solution of the differential equation \eqref{DIFFEQ} satisfying the conditions $y(0)=0$ and $y'(0)=1$. Define the function $T$ on the interval $[0,1)$ by
    $$T(x):=g'(x)-\dfrac{y'(x)g(x)}{y(x)},$$
    which is obviously continuous since $y$ only vanishes at $0$ and $g(0)=0$.
    Since $q_\alpha \in \mathcal{P}((1+\alpha)/2)$, we have
    $$\lim_{x\to1^-}\dfrac{xy'(x)}{y(x)}=\lim_{x\to1^-}\dfrac{y'(x)}{y(x)} \geq \dfrac{1+\alpha}{2}.$$
    Then for given $\epsilon>0$, there exists a $\delta>0$ such that for all $r \in (1-\delta,1),$ we have
    \begin{equation}\label{LIMINEQEPSILON}
        \dfrac{ry'(r)}{y(r)}>\dfrac{1+\alpha}{2}-\epsilon.
    \end{equation}
    Note that due to the continuity of the function $T$, the integral $\int_0^r T^2(x) dx$ exists for any $r \in (1-\delta,1)$. Now from the fact that $\int_0^r T^2(x) dx \geq 0$, we have
    \begin{equation}\label{INTEQ1}
        \int_0^r (g'(x))^2dx - 2 \int_0^r \dfrac{g'(x)y'(x)g(x)}{y(x)} dx+ \int_0^r\left(\dfrac{y'(x)g(x)}{y(x)}\right)^2 dx \geq 0.
    \end{equation}
    Now note that 
    $$\dfrac{d}{dx} \left(\dfrac{y'(x)}{y(x)}\right) = \dfrac{y''(x)}{y(x)} - \left(\dfrac{y'(x)}{y(x)}\right)^2,$$
    which gives

    $$\int_0^r\left(\dfrac{y'(x)g(x)}{y(x)}\right)^2 dx = \int_0^r g^2(x) \left[\dfrac{y''(x)}{y(x)}-\dfrac{d}{dx} \left(\dfrac{y'(x)}{y(x)}\right)\right]dx.$$
 Then using the integration by parts, we obtain
 $$\int_0^r\left(\dfrac{y'(x)g(x)}{y(x)}\right)^2 dx = \int_0^r \dfrac{g^2(x)y''(x)}{y(x)}dx-\dfrac{g^2(r)y'(r)}{y(r)}+2 \int_0^r \dfrac{g'(x)y'(x)g(x)}{y(x)} dx,$$
 which reduce the inequality \eqref{INTEQ1} to
 $$ \int_0^r (g'(x))^2dx + \int_0^r \dfrac{g^2(x)y''(x)}{y(x)}dx-\dfrac{g^2(r)y'(r)}{y(r)}\geq0.$$
Now using the fact $y''(x)/y(x)=-q_\alpha(x)$ and along with the inequality \eqref{LIMINEQEPSILON}, we obtain that
$$\int_0^r (g'(x))^2dx \geq \int_0^r g^2(x)q_\alpha(x)dx+\left(\dfrac{1+\alpha}{2}-\epsilon\right)\dfrac{g^2(r)}{r},$$
which gives our desired inequality.
\end{proof}

\medskip
We are restating the following lemma from \cite{ARO18}, which plays an important role in proving Theorem \ref{MAINT1}.
\begin{lemma}\label{LEMMAFORFUNCTIONRATIO}
    Let $f \in \mathcal{B}$ be a locally univalent function with the Schwarzian derivative $S_f=2p$, where $p$ is analytic on $\mathbb{D}$. Also, assume $f=u/v$, where $u$ and $v$ are linearly independent solutions of the differential equation $y''+py=0$. Then 
    $$f\in \mathcal{BC}(\alpha) \Leftrightarrow v \in \mathcal{S}^*((1+\alpha)/2).$$ 
\end{lemma}

 Now, we give the following lemma, which discusses a sufficient condition for a non-negative continuous function defined on $[0,1)$ to be included in the class $\mathcal{P}(\alpha)$. Since the proof is similar to that of \cite[Theorem 1]{CHUA11}, we omit it.

\begin{lemma}\label{OURLEMMA2}
    Let $q:[0,1)\to \mathbb{R}$ be a continuous non-negative function with 
    $$\int_0^1 q(x) dx \leq c,$$
where $c\leq 1$.    Then $q \in \mathcal{P}(\alpha)$ for every $\alpha \leq1-c$. For any $\alpha>1-c$, $q$ need not belong to the class $\mathcal{P}(\alpha)$.
\end{lemma}

\medskip

\section{Proof of the Main Results}
We begin with the proof of Theorem  \ref{MAINT1}.
\subsection{Proof of Theorem \ref{MAINT1}.}
    Let us take $p=S_f/2$, which is clearly analytic on $\mathbb{D}$. Then by  Lemma \ref{DURENTHEOREM1},
    $f$ is of the form \eqref{U/V} and by Lemma \ref{LEMMAFORFUNCTIONRATIO}, to prove $f \in \mathcal{BC}(\alpha)$ it is enough to show  the function $v$ is starlike of order $(1+\alpha)/2$.
Now consider the differential equation
\begin{equation*}\label{DEINPROOF}
    v''+pv=0,
\end{equation*}
with initial conditions  $v(0)=0$ and $v'(0)=1$.
Then we have
\begin{equation}\label{EQ1THM2.10}
    v''\overline{v}+p|v|^2=0.
\end{equation}
Fix $z_0=re^{it_0}$, where $t_0 \in [0,2\pi)$ and $r \neq 0$. Then integrate equation \eqref{EQ1THM2.10} through the ray from $0$ to $z_0$, we have

$$\overline{v(z_0)} v'(z_0) - \int_0^r e^{-it_0} \overline{v'(\rho e^{it_0})}  v'(\rho e^{it_0})  d\rho+ \int_0^re^{it_0} p(\rho e^{it_0}) |v(\rho e^{it_0})|^2 d\rho =0.$$
Now multiplying throughout by $re^{it_0}$, we obtain

$$z_0 |v(z_0)|^2 \frac{v'(z_0)}{v(z_0)}- r\int_0^r |v'(\rho e^{it_0})|^2 d\rho + r\int_0^r e^{2it_0} p(\rho e^{it_0})|v(\rho e^{it_0})|^2  d\rho=0.$$
It follows that
$$|v(z_0)|^2 \real \left(\dfrac{z_0v'(z_0)}{v(z_0)}\right)=r \int_0^r |v'(\rho e^{it_0})|^2 d\rho-r\int_0^r |v(\rho e^{it_0})|^2 \real \left(e^{2it_0} p(\rho e^{it_0})\right)d\rho.$$
From the hypothesis, we have
$$\real \left(e^{2it_0} p(\rho e^{it_0})\right) \leq |e^{2it_0} p(\rho e^{it_0})| = |p(\rho e^{it_0})| = |S_f(\rho e^{it_0})|/2 \leq q_\alpha(\rho).$$
Then we obtain
\begin{equation}\label{INEQFROMSF}
    |v(z_0)|^2 \real \left(\dfrac{z_0v'(z_0)}{v(z_0)}\right) \geq r \int_0^r |v'(\rho e^{it_0})|^2 d\rho - r\int_0^r q_\alpha(\rho)|v(\rho e^{it_0})|^2 d\rho. 
\end{equation}
Now take arbitrary $\epsilon>0$. Then by applying  Lemma \ref{OURLEMMA1} for the functions $\eta$ and $\zeta$, there exists $\delta>0$ such that for $r \in (1-\delta,1)$, we have
\begin{equation}\label{INEQFROMLEMMA}
    r \int_0^r |v'(\rho e^{it_0})|^2 d\rho \geq |v(z_0)|^2 \left(\dfrac{1+\alpha}{2} -\epsilon \right)+r \int_0^r q_\alpha(\rho)|v(\rho e^{i t_0})|^2 d\rho,
\end{equation}
where $v=\eta+i\zeta$.
Applying the inequality (\ref{INEQFROMLEMMA}) in (\ref{INEQFROMSF}) to obtain
$$\real \left(\dfrac{z_0v'(z_0)}{v(z_0)}\right) \geq \dfrac{1+\alpha}{2} -\epsilon.$$
Since the function $\real (zv'(z)/v(z))$ is a harmonic function, it attains its minimum on the boundary. Thus, for all $z$ such that $|z|\leq r$, we have
$$\real \left(\dfrac{zv'(z)}{v(z)}\right) \geq \dfrac{1+\alpha}{2} -\epsilon.$$
Making $\epsilon \to 0$ and since $r$ is arbitrary in $(0,1)$, we get 
$$\real \left(\dfrac{zv'(z)}{v(z)}\right) \geq \dfrac{1+\alpha}{2}, \text{ for }|z| < 1,$$
which implies $v \in \mathcal{S}^*((1+\alpha)/2).$
Hence, we obtained $f$ is meromorphically convex of order $\alpha$.

\medskip

Next we will show that the constant $(1+\alpha)/2$ is the best possible. For a fixed $n \in \mathbb{N}$ and $\beta \in [0,1)$, we define the following function
$$q_\beta(z) := (1- \beta) (n+1)z^n, \, z \in \mathbb{D}.$$ 
Integrating to obtain
$$\int_0^1 q_\beta(t)dt =1- \beta,$$
which shows that $q_\beta \in \mathcal{P}(\beta)$. Now let $y$ be the solution of the differential equation
\begin{equation}\label{THMDIFFEQ}
    y''+q_\beta y =0, \, y(0)=0, \, y'(0)=1.
\end{equation}
Fix $\omega \neq 0 \in \mathbb{D}$ and define the function
$$f(z):= -\int_\omega^z y(s)^{-2}ds,$$
where the integral is taken over an arbitrary path from $\omega$ to $z$.
It is clear that $$f'(z) = \dfrac{-1}{y^2(z)},$$
and
$$f''(z) = \dfrac{2y'(z)}{y^3(z)}.$$
Then the Schwarzian derivative of $f$ is

$$S_f(z) = \left(\dfrac{-2y'(z)}{y(z)}\right)'- \dfrac{1}{2} \left(\dfrac{-2y'(z)}{y(z)}\right)^2=-2\dfrac{y''(z)}{y(z)} =2q_\beta(z) = 2(1-\beta)(n+1)z^n,$$
which gives
$|S_f(z)|=2q_\beta(|z|).$ 

\medskip
Let $0<a<x<1$ and take $\epsilon>0$. For each $n \in \mathbb{N}$, let $y_n$ be the solution of the differential equation \eqref{THMDIFFEQ}. Since for small values of $x$, the term $x^{n+1}$ tends to $0$ quickly, and in this case, $y_n''$ is approximately zero. Hence the solution $y_n$ is linear near the point zero, and as a result, there exists $N \in \mathbb{N}$ such that for all $n \geq N$, we have $y_{n}(a)>a-\epsilon$. Also note that for any $n \in \mathbb{N}$, we have
$$\int_a^x q_\beta(t)dt = (1-\beta) (x^{n+1}-a^{n+1}).$$
Then there exists $M \in \mathbb{N}$ such that for all $n\geq M$, we have
$$\int_a^x q_\beta(t)dt > (1-\beta) (x^{n+1}-\epsilon).$$
Let $N_1 = \max\{N, M\}$ and call the solution $y_{N_1}=y$. By proceeding as similar in the proof of Lemma \ref{OURLEMMA2}, we get $y$ is increasing on $(0,1).$ Hence, we have

$$y'(x) = 1- \int_0^x q_\beta(t)y(t) dt < 1-(a-\epsilon) \int_a^x q_\beta(t) dt < 1-(1-\beta)(a-\epsilon)(x^{n+1}-\epsilon),$$
which gives
$$\lim_{x\to 1^-} \dfrac{y'(x)}{y(x)} \leq \dfrac{1}{a-\epsilon} - (1-\beta)(1-\epsilon).$$
It is clear that we can make the bound less than $\beta$ by tending $a \to 1$ and $\epsilon \to 0$. Then there exists $x_0 \in (0,1)$ such that
$$\dfrac{x_0 y'(x_0)}{y(x_0)}\leq\beta.$$
Now, for any $\beta <(1+\alpha)/2$, we have
$$1+x_0 \dfrac{f''(x_0)}{f'(x_0)} = 1-\dfrac{2x_0y'(x_0)}{y(x_0)}\geq1-2\beta>-\alpha,$$
which implies that $f$ is not meromorphically convex function of order $\alpha$. Hence, the proof is completed.
\hfill{$\Box$}
\subsection{Proof of Theorem \ref{MAINT5}}
Let $g \in \mathcal{BC}_{I}(\alpha)$.
\begin{enumerate}[label=(\roman*)]
    \item The result is obvious from the inclusion property $$\mathcal{C}(\alpha) \subset \mathcal{C}(\beta), \text{ for } \alpha>\beta.$$ 
    \item Since $g \in \mathcal{BC}_{I}(\alpha)$ and $g=1/f$, where $f \in \mathcal{C}(\alpha).$ Define the function $$h(z)=\lambda g(z), \, \lambda \neq 0. $$
    Since the class $\mathcal{C}(\alpha)$ is invariant under the non-zero scalar multiplication, we have  $f/\lambda$ also in $\mathcal{C}(\alpha)$, which implies $h \in \mathcal{BC}_{I}(\alpha)$.
    \item Since $g \in \mathcal{BC}_{I}(\alpha)$, we have $1/g \in \mathcal{C}(\alpha)$. The fact $\mathcal{C}(\alpha) \subset \mathcal{S}^*$, gives  $1/g \in \mathcal{S}^*$. Then by the relation 
\eqref{fand1/f}, we obtain $g \in \mathcal{BS}^*(0)$. Now consider the function $h$ defined on $\mathbb{D}$ by $$h(z) := z+\dfrac{1}{z}-2.$$
Note that $1/h$ is nothing but the well-known Koebe function, which is clearly not in the class $\mathcal{C}(\alpha)$ for any $\alpha \in [0,1)$. So $h \notin \mathcal{BC}_I(\alpha)$ for any $\alpha \in [0,1)$.  But note that $1/h$ is in the class $\mathcal{S}^*$ and by the relation \eqref{fand1/f}, we have $h \in \mathcal{BS}^*(0)$.

\item Since $g \in \mathcal{BC}_{I}(\alpha)$, there exists $f \in \mathcal{C}(\alpha)$ such that for all $z \in \mathbb{D}\setminus\{0\}$,
$$g(z) = \dfrac{1}{f(z)}.$$
Clearly, $g$ is univalent.  
Simple computations give
$$f''(z) = \dfrac{2(g'(z))^2}{g^3(z)}-\dfrac{g''(z)}{g^2(z)}.$$
Then clearly we obtain for all $z \in \mathbb{D}\setminus\{0\}$,
\begin{equation}\label{DIFFEQN6}
    \dfrac{zf''(z)}{f'(z)} = \dfrac{zg''(z)}{g'(z)}-\dfrac{2zg'(z)}{g(z)}.
\end{equation}
Since $f \in \mathcal{C}(\alpha)$ we have the inequality, 
$$1+\real \left(\dfrac{zf''(z)}{f'(z)}\right)\geq\alpha.$$
Then the equation \eqref{DIFFEQN6} gives our desired result. Conversely, if $g$ is univalent and satisfies (\ref{EQT3}), then by taking $f=1/g$, the result is obvious.
\end{enumerate} \hfill{$\Box$}
\subsection{Proof of Theorem \ref{MAINT2}.}
    Let $g \in \mathcal{M}.$ Then $g=1/f$ for some $f \in \mathcal{S}$. From  \cite[p. 32]{DUR83} and the relation \eqref{DIFFEQN6}, we get

    \begin{equation*}
    \begin{split}
        \real \left(\dfrac{2zg'(z)}{g(z)}-\dfrac{zg''(z)}{g'(z)}\right)  &= \real \left(-\dfrac{zf''(z)}{f'(z)} \right)\\ 
        &=\real \left(\dfrac{2r^2}{1-r^2}-\dfrac{zf''(z)}{f'(z)} \right)-\dfrac{2r^2}{1-r^2}\\
        &\leq \left| \dfrac{2r^2}{1-r^2}-\dfrac{zf''(z)}{f'(z)} \right| - \dfrac{2r^2}{1-r^2} \\&\leq\dfrac{4r-2r^2}{1-r^2}.
    \end{split}
    \end{equation*}
    Now consider the polynomial $P_\alpha(x)$ as in the equation (\ref{POLYRADIUS}),
    which is trivially continuous and increasing on $(0,1)$. Also note that $P_\alpha(0)=\alpha-1<0$ and $P_\alpha(1)=2>0$, which gives the existence of exactly one root of $P_\alpha(x)$ in the interval $(0,1)$, say $r_\alpha$. Then for all $z$ such that $|z| \leq r_\alpha$,  we have
    $$\real \left(\dfrac{2zg'(z)}{g(z)}- \dfrac{zg''(z)}{g'(z)}  \right) \leq 1-\alpha.$$
    Then by the analytic characterization \eqref{EQT3}, we proved $g$ is meromorphically inverse convex of order $\alpha$ on the disk of radius $r_\alpha$.

\medskip

    For proving the sharpness of the radius, consider the function $g$ defined by 
    $$g(z) := z+\dfrac{1}{z}-2.$$
    Note that the reciprocal of the function $g$ is nothing but the well-known Koebe function, which is trivially not a convex function. So $g$ cannot be in the class $\mathcal{BC}_I(\alpha)$.
    But simple computations give
    $$\dfrac{2zg'(z)}{g(z)}-\dfrac{zg''(z)}{g'(z)} = -\dfrac{4z+2z^2}{1-z^2}.$$
    Now choose $r \in (0,1)$ such that $r>r_\alpha$. Then we have $P(r)>0$ and as a result
    $$\dfrac{4r-2r^2}{1-r^2}>1-\alpha,$$
    which shows that the bound $r_\alpha$ is sharp for a suitable rotation of the function $g$. The radius $r_\alpha$ is called {\em radius of inverse convexity} for the class $\mathcal{M}$. \hfill{$\Box$}
\subsection{Proof of Theorem \ref{MAINT4}}
 Let $g \in \mathcal{BC}_{I}(\alpha)$. Then there exists a function $f \in \mathcal{C}(\alpha)$ such that $g=1/f$. Since 
$$\real\left(1+\dfrac{zf''(z)}{f'(z)}\right)\geq\alpha,$$
we have
\begin{equation}\label{EQ1T4}
    \real\left(\dfrac{zh'(z)}{h(z)}\right) \geq \alpha,
\end{equation}
where $h(z)=zf'(z)$ and this implies $h \in \mathcal{S}^*(\alpha)$. That is, 
$$h(z) = -\dfrac{zg'(z)}{g^2(z)} \in \mathcal{S}^*(\alpha).$$
Note that if we take $h_1=1/h$, {then} the equation (\ref{EQ1T4}) will be equivalent to
$$\real\left(\dfrac{zh_1'(z)}{h_1(z)}\right)\leq-\alpha,$$
which gives $h_1 \in \mathcal{BS}^*(\alpha)$. This completes the proof. \hfill{$\Box$}


\subsection*{Acknowledgement}
The work of the first author is
	supported by SERB-SRG (SRG/2023-\\/001938) and Faculty Research Grant FRG (NITC/(R\&C)/2023-2024/FRG/Phase III/(68)). The work of second author is supported by an INSPIRE fellowship of the Department of Science and Technology, Govt. of India (DST/INSPIRE Fellowship/2021/IF210612).

\subsection*{Conflict of Interests}
The authors declare that there is no conflict of interests regarding the publication of this paper.

\subsection*{Data Availability Statement}
The authors declare that this research is purely theoretical and does not associate with any datas.


\begin{thebibliography}{99}

\bibitem{AHA14} {D. Aharanov and U. Elias}, {Sufficient conditions for univalence of analytic functions}, {\em J. Anal.} {\bf 22} (2014), 1--11.



\bibitem{ARO18} {V. Arora and S. K. Sahoo}, {Meromorphic functions with small Schwarzian derivative}, {\em Stud. Univ. Babe-Bolyai Math.} {\bf 63} (3) (2018), 355--370.

\bibitem{BEC72} {J. Becker}, {L\"ownersche Differentialgleichung und quasikonform fortsetzbare schlichte Funktionen}, {\em J. Reine Angew.} {\bf 255} (1972), 23--43.



\bibitem{DE85}{L. de Brange}, {A proof of the Bieberbach conjecture}, {\em Acta Math.} {\bf 154}  (1985), 137--152.

\bibitem{CHU11} {M. Chuaqui,  P. Duren, and B. Osgood}, {Schwarzian derivatives of convex mappings}, {\em Ann. Acad. Sci. Fenn. Math.} {\bf36}(2) (2011), 449--460.

\bibitem{CHUA11} {M. Chuaqui,  P. Duren, and B. Osgood}, {On a theorem of Haimo regarding concave mappings}, {\em  Ann. Univ. M. Curie-Sklodowska} {\bf 65} (2011), 17--28.

\bibitem{CLU59} {J. Clunie}, {On meromorphic schlicht functions}, {\em J. London Math. Soc.} {\bf 34} (1959), 215--216.

\bibitem{DUR83}{P. L. Duren}, {\em Univalent function}, Springer-Verlag, New York, 1983.



\bibitem{GAB55} {R. F. Gabriel}, {The Schwarzian derivative and convex functions}, {\em Proc. Amer. Math. Soc.} {\bf 24} (1955), 58--66.


\bibitem{GOO83} {A. W. Goodman}, {\em Univalent Functions, Vols. I and II}, Polygonal Publishing House, Washington, 1983.


\bibitem{HAI56} {D. T. Haimo}, {A note on convex mappings}, {\em Proc. Amer. Math. Soc.} {\bf 7} (1956), 423--428.





\bibitem{HIL49}{E. Hille}, {Remark on a paper by Zeev Nehari}, {\em Proc. Amer. Math. Soc.} {\bf 55} (1949), 552--553.

\bibitem{KIM22} {J. A. Kim and T. Sugawa}, {Geometric properties of functions with small Schwarzian derivatives}, {POSTECH Korea}, 2004. 

\bibitem{LEH77} {O. Lehto}, {Domain constants associated with Schwarzian derivative}, {\em  Comment. Math. Helv.} {\bf 52} (1977), 603--610.

\bibitem{LEH87} {O. Lehto}, {\em Univalent functions and Teichm\"uller spaces}, {Springer-Verlag}, {Newyork}, 1987.

\bibitem{MAM94} {W. C. Ma and D. Minda}, {A unified treatment of some special classes of univalent functions}, in {\em Proceedings of the Conference on Complex Analysis}(Tianjin, 1992), 157--169, Conf. Proc. Lecture Notes Anal. I, Int. Press, Cambridge.

\bibitem{MIL70} {J. Miller}, {Convex meromorphic mappings and related functions}, {\em Proc. Amer. Math. Soc.} {\bf 25} (1970), 220--228.

\bibitem{MIL00} {S. S. Miller and P. T. Mocanu}, {\em Differential subordinations: theory and applications}, {Dekker, New York, 2000}.


\bibitem{NEH49}{Z. Nehari}, {The Schwarzian derivative and schlicht functions}, {\em Bull. Amer. Math. Soc.} {\bf 55} (1949), 545--551.

\bibitem{NEH54} {Z. Nehari}, {Some criteria of unvialence}, {\em Proc. Amer. Math. Soc.} {\bf 5} (1954), 700--704.

\bibitem{NEH76} {Z. Nehari}, {A property of convex conformal maps}, {\em J. Analyse Math.} {\bf 30} (1976), 390--393.

\bibitem{NEH79} {Z. Nehari}, {Univalence criteria depending on the Schwarzian derivative}, {\em Illinois J. Math.} {\bf 23} (1979), 343--351.

\bibitem{NUN01} {M. Nunokawa and O. P. Ahuja}, { On meromorphic starlike and convex functions}, {\em Indian J. Pure Appl. Math.} {\bf 32}(7) (2001), 1027--1032.

\bibitem{OSG98} {B. Osgood}, {Old and new on the Schwarzian derivative}, {\em Quasiconformal mappings and analysis (Ann Arbor, MI, 1995), Springer, New York} (1998), 275-308.



\bibitem{POM63} { Ch. Pommerenke}, { On meromorphic starlike functions}, {\em Pacific J. Math.} {\bf 13} (1963), 221--235.

\bibitem{POM75} { Ch. Pommerenke},  {\em Univalent functions}, {Vandenhoeck and Ruprecht}, {G\"ottingen}, 1975.

\bibitem{PON14} {S. Ponnusamy, S. K. Sahoo, and  T. Sugawa}, {Radius problems associated with pre-Schwarzian and Schwarzian derivatives}, {\em Analysis} {\bf 34} (2014), 163--171.

\bibitem{ROB36}{M. S. Robertson}, {On the theory of univalent functions}, {\em Ann. of Math.} {\bf 37} (1936), 374--408.



\bibitem{TIT39} {E. C. Titchmarsh}, {\em The Theory of Functions}, Oxford University Press, London, 1939.

\bibitem{TOT10} {A. Totoi}, {On inverse convex meromorphic functions}, {\em Stud. Univ. Babe\c s-Bolyai Math.} {\bf 55} (2010), 167--174.

\bibitem{YAM79} {S. Yamashita}, {Inequalities for the Schwarzian derivative}, {\em Indian Univ. Math. J.} {\bf 28} (1979), 131--135.


\end{thebibliography}
\end{document}